\documentclass[11pt]{article}
\usepackage{mathrsfs}
\usepackage{amsfonts}
\usepackage[centertags]{amsmath}
\usepackage{amssymb}
\usepackage{amsthm}
\usepackage{enumerate}
\usepackage{graphicx}
\usepackage{appendix}
\usepackage{amsmath}
\usepackage{ragged2e}
\usepackage{multirow}
\renewcommand{\raggedright}{\leftskip=0pt \rightskip=0pt plus 0cm}
\newtheorem{thm}{Theorem}[section]
\newtheorem*{thm*}{Theorem}

\newtheorem{lem}[thm]{Lemma}

\newtheorem{exam}[thm]{Example}

 \def\lg{\langle} \def\rg{\rangle}
\def\nd{\mathrel{\bigm|\kern-.7em/}}

\def\f{\noindent}

\def\mod{\hbox{\rm mod }}

\def\qed{\hfill $\Box$}

\begin{document}

\title{A note on the bound for the class of certain nilpotent groups\thanks{This
work was supported by NSFC (No. 12271318).}}

\author{Jixia Gao$^{1}$, Haipeng Qu$^{2}$\thanks{corresponding author}\\
\footnotesize\em $1$. School of Mathematics and Statistics, Lanzhou University,\\
\footnotesize\em Lanzhou, Gansu 730000, P. R. China\\
\footnotesize\em e-mail: gaojx19@lzu.edu.cn\\
\footnotesize\em $2.$ School of Mathematics and Computer Science, Shanxi Normal University, \\
\footnotesize\em Taiyuan,  Shanxi 030032, P. R. China\\
\footnotesize\em e-mail: orcawhale@163.com}
\date{}
\maketitle

\begin{abstract}
Assume $G$ is a nilpotent group of class $>3$ in which every proper subgroup has class at most $3$. In this note, we give the exact upper bound of class of $G$.

\medskip


\medskip

\end{abstract}
\baselineskip=16pt
\section{Introduction}
Evans and Sandor in \cite{ES} addressed the following

\smallskip
{\bf Question 1.} Let $n$ be a positive integer. Does there exist a nilpotent group $H_n$ of class
$2n$ in which every proper subgroup has class at most $n$?
\smallskip

If such a group $H_n$ exists it must be finite (see \cite[Theorem 3.3]{NW} or \cite[Theorem 1]{Za})
and $H_n$ must be a $p$-group for some a prime $p$ (see \cite[Theorem 1.2]{CK}). Thus Question 1 can be restated as follows.

\smallskip
{\bf Question 1$^{\prime}$.} Let $n$ be a positive integer. Does there exist a finite $p$-group $H_n$ of class
$2n$ in which every proper subgroup has class at most $n$?
\smallskip

Evans and Sandor in \cite{ES} proved that the answer to Question 1$^{\prime}$ is `yes' when $p>2n$ and $n\geqslant3$. A nature question is: when $p\leqslant 2n$, what can be said about Question 1$^{\prime}$? If $n=1$, then $H_1$ is a finite nonabelian $p$-group in which every proper subgroup is abelian. $H_1$ was classified by R\'{e}dei in \cite{minimal nonabelian}. For each prime $p$ his classification contains infinitely many $p$-groups. Hence the answer to Question 1$^{\prime}$ is `yes' if $n=1$.
In contrast the answer is `no' for $n=2$ (see \cite[Corollary 1 to Theorem 1]{p>5} or \cite[Theorem 3.3]{B2}). If $n=3$, then the answer to Question 1$^{\prime}$ is `yes' when $p\geqslant5$ (see \cite[Theorem 4]{p>5}). However, we prove in this note that if $n=3$ then the answer to Question 1$^{\prime}$ is `no' when $p=2$ or $3$. In this case, we prove $c(G)\leqslant 5$ for a $p$-group $G$ in which every proper subgroup has class at most $3$. And examples are provided  to show this bound is best possible. Based on such observation, we propose the following

\smallskip
{\bf Conjecture.} Let $n$ be a positive integer with $n\geqslant 4$ and $p=2$ or $3$. Then $c(G)<2n$ for a $p$-group $G$ in which every proper subgroup has class at most $n$.

\section{Preliminaries}

Let $G$ be a finite $p$-group. We use $c(G)$, $\exp(G)$ and $d(G)$ to denote the nilpotency class, the exponent and the minimal cardinality of generating set of $G$, respectively. We use ${\rm C}_{p^n}$ and ${\rm C}_{p^n}^m$ to denote the cyclic group of order $p^n$ and the direct product of $m$ copies of ${\rm C}_{p^n}$, respectively. For a nilpotent group $G$, let
$$G=K_1(G)> G'=K_2(G)> K_3(G)> \cdots> K_{c+1}(G)=1$$
be the lower central series of $G$, where $c=c(G)$ and $K_{i+1}(G)=[K_i(G),G]$ for $1\leqslant i\leqslant c$. A finite group $G$ is said to be metabelian if $G''=1$.

\begin{lem}\label{b1}{\rm (\cite[Proposition 2.1.1]{Qu})}~Let $G$ be a finite group and $x,y,z\in G$. Then

\vspace{0.1cm}

$(1)~[x,yz]=[x,z][x,y][x,y,z];$

\vspace{0.1cm}

$(2)~[xy,z]=[x,z][x,z,y][y,z];$

\vspace{0.1cm}

$(3)~[x,y]^{-1}=[y,x]=[x,y^{-1}]^y=[x^{-1},y]^x$.
\end{lem}

\begin{lem}\label{p2}{\rm (\cite[Proposition 2.1.3]{Qu})}~Let $G$ be a group. n is a positive integer and $a_1,\cdots,a_i,$$\cdots, a_n,b_i\in G,~{\rm where}~1\leqslant i\leqslant n$. Then

\vspace{0.1cm}

$(1)~[a_1,\cdots,a_ib_i,\cdots, a_n]\equiv [a_1,\cdots,a_i,\cdots,a_n][a_1,\cdots,b_i,\cdots,a_n]~(\mod K_{n+1}(G));$

\vspace{0.1cm}

$(2)~[a_1^{i_1},\cdots,a_n^{i_n}]\equiv [a_1,\cdots,a_n]^{i_1\cdots i_n}~(\mod K_{n+1}(G)),$ where $i_1,\cdots,i_n$ are any integers.
\end{lem}

\begin{lem}\label{p3}{\rm (\cite[III, 1.11 Hilfssatz]{Hu})}~Let $G$ be a finite group and $G=\langle M\rangle$. Then
$$K_n(G)=\langle [x_1,\cdots,x_n],K_{n+1}(G)\mid x_i\in M\rangle.$$
\end{lem}

\begin{lem}\label{l3}{\rm (\cite[Lemma 3.2]{B2})}~Assume $G$ is a finite $p$-group. If $d(G)=2$ and $K_{5}(G)=1$, then $G$ is metabelian.
\end{lem}

\begin{lem}\label{p1}{\rm (\cite{Xu})}~Assume $G$ is a metabelian group and $a, b \in G$, $c \in G'$. Then

$(1)$ $[c,a,b]=[c,b,a];$

$(2)$ For any positive integers $m$ and $n$,
$$[a^m,b^n]=\prod\limits_{i=1}^{m}\prod\limits_{j=1}^{n}[ia,jb]^{{m\choose i}{n\choose j}},$$
where $[ia,jb]=[a,b,\underbrace{a,\cdots,a}_{i-1},\underbrace{b,\cdots,b}_{j-1}].$
\end{lem}

\begin{lem}\label{d1}{\rm (\cite[Theorem 3.5]{B2})}~Assume $G$ is a finite group in which itself is not of nilpotency class $\leqslant n$ and all its proper subgroups are of nilpotency class $\leqslant n$. Then $c(G)\leqslant\frac{dn}{d-1}$, where $d=d(G)$.
\end{lem}

\section{The Main Results and Their Proofs}

For convenience, denoted by $\mathcal{P}_n$  the set of all $p$-groups whose class $>n$ in which every proper subgroup has class at most $n$. In this section, the main theorem we will prove is
\begin{thm*}
Assume $G$ is a finite $p$-group and $G\in \mathcal{P}_3$. If $p=2$ or $3$,  then $c(G)\leqslant5$.
\end{thm*}
Obviously, if $d(G)\geqslant3$ for the group $G$ in Theorem, by Lemma \ref{d1}, the Theorem is always true. So assume $d(G)=2$ for the group $G$ in Theorem as follows.

\begin{lem}\label{c1} Assume $G$ is a finite $p$-group and $G\in \mathcal{P}_n$. For any $x,y \in G$, let $t=[x,y,z_1,\cdots,z_s]$, where $s\geqslant n$ and $z_i\in \{x,y\}$ for $1\leqslant i\leqslant s$. Then $t=1$
if there are at least $n$  $x$'s  or at least $n$  $y$'s in $\{z_{1},\cdots,z_{s}\}$.
\end{lem}

\begin{proof}
Let $H=\langle x,G'\rangle$. Then $H\lhd G$. Since $G\in \mathcal{P}_n$,
$c(H)\leqslant n$. Thus there exists a lower central series of $H$:
$$H=K_1(H)\geq K_2(H)\geq K_3(H)\geq\cdots\geq K_{n+1}(H)=1.$$
Then $K_i(H)~ char ~H$, where $1\leqslant i\leqslant n+1$. Moreover, $K_i(H)\lhd G$. Notice that for any $h\in K_i(H)$,
$$[h,x]\in K_{i+1}(H) ~{\rm and}~ [h,y]=h^{-1}h^y\in K_i(H).$$
Thus if there are at least $n$  $x$'s  in $\{z_{1},\cdots,z_{s}\}$, $t\in K_{n+1}(H)=1$.
Let $K=\langle y,G'\rangle$. Similar to the above, if there are at least $n$  $y$'s  in $\{z_{1},\cdots,z_{s}\}$, then $t=1$.
\end{proof}

\begin{lem}\label{t1} Assume $G$ is a finite $p$-group with $d(G)=2$. Then for any $a,b\in G$,
$$[a,b,a,b]\equiv[a,b,b,a]~(\mod K_5(G)).$$
\end{lem}

\begin{proof}
Since $d(G/K_5(G))=2$ and $K_5(G/K_5(G))=K_5(G)/K_5(G)=1$, $G/K_5(G)$ is metabelian by Lemma~\ref{l3}. Thus, by Lemma~\ref{p1}$(1)$,~for any $a,b\in G$,
$$[a,b,a,b]\equiv[a,b,b,a]~(\mod K_5(G)).$$
\end{proof}

\begin{lem}\label{l6} Assume $G$ is a finite $p$-group and $G\in \mathcal{P}_3$, $n$ is a positive integer. If $x,y\in G$ and $\langle x,y\rangle< G$, then
$$[x^n,y]=[x,y]^n[x,y,x]^{(^n_2)} ~{\rm and} ~[x,y^n]=[x,y]^n[x,y,y]^{(^n_2)}.$$
\end{lem}
\begin{proof}
Since $G\in \mathcal{P}_3$ and $\langle x,y\rangle< G$, $c(\langle x,y\rangle)\leqslant3$.
Thus $\langle x,y\rangle$ is metabelian and $[x,y,y,y]=[x,y,x,x]=1$.
Therefore, by Lemma~\ref{p1}$(2)$,
$$[x^n,y]=[x,y]^n[x,y,x]^{(^n_2)} ~{\rm and}~ [x,y^n]=[x,y]^n[x,y,y]^{(^n_2)}.$$
\end{proof}

By Lemma \ref{d1}, $c(G)\leqslant 6$ if $G\in \mathcal{P}_3$. Hence $K_7(G)=1$. The fact that $K_7(G)=1$ will be used frequently in the proof of Lemma \ref{t3}, \ref{t2}, \ref{l31} and the Theorem. For simplicity, we will not mention it explicitly.


\begin{lem}\label{t3} Assume $G$ is a finite $p$-group and $G=\langle a,b\rangle$. If $G\in \mathcal{P}_3$, then
$$K_{6}(G)=\langle[a,b,a,b,a,b]\rangle\lesssim {\rm C}_p.$$
\end{lem}

\begin{proof}
It suffices to show that $K_{6}(G)\leq\langle[a,b,a,b,a,b]\rangle$ and $[a,b,a,b,a,b]^p=1$.
Let $S=\langle[a,b,a],a^p,b\rangle$. Then $S<G$. Since $G\in \mathcal{P}_3$, $c(S)\leqslant 3$. Thus $[[a,b,a],b,a^p,b]=1$. Moreover, $[a,b,a,b,a,b]^p=1$ by Lemma \ref{p2}$(2)$. Now we prove $K_{6}(G)\leq\langle[a,b,a,b,a,b]\rangle$.

It follows from Lemma \ref{p3} that
$$K_6(G)=\langle[a,b,x_1,x_2,x_3,x_4]\mid x_i\in\{a,b\} ~\mbox{for} ~1\leqslant i\leqslant4\rangle.$$
By Lemma~\ref{c1},~we only need to consider the case that there are two $a$'s and two $b$'s for $x_i$'s.
Let $H=\langle[a,b],ab\rangle$. Then $H<G$. Since $G\in \mathcal{P}_3$, $c(H)\leqslant 3$. Thus
$$[[a,b],ab,ab,ab]=1, ~{\rm and~hence} ~[a,b,ab,ab,ab,a]=1.$$
By Lemma \ref{p2}$(1)$ and Lemma~\ref{c1}, we have
\begin{equation}\label{1-1}
  [a,b,b,b,a,a][a,b,b,a,b,a][a,b,a,b,b,a]=1.
\end{equation}
Let $K=\langle[a,b,b],ab\rangle$. Then $K<G$. Since $G\in \mathcal{P}_3$, $c(K)\leqslant 3$. Thus
$$[[a,b,b],ab,ab,ab]=1.$$
By Lemma \ref{p2}$(1)$ and Lemma~\ref{c1}, we have
\begin{equation}\label{1-2}
  [a,b,b,b,a,a][a,b,b,a,b,a][a,b,b,a,a,b]=1.
\end{equation}
Similarly, let $T=\langle[a,b,a],ab\rangle$. We have
\begin{equation}\label{1-3}
  [a,b,a,a,b,b][a,b,a,b,a,b][a,b,a,b,b,a]=1.
\end{equation}
By $(\ref{1-1})$ and $(\ref{1-2})$,
$$[a,b,a,b,b,a]=[a,b,b,a,a,b].$$
Thus, by Lemma~\ref{t1},
\begin{equation}\label{1-4}
  [a,b,b,a,b,a]=[a,b,a,b,b,a]=[a,b,b,a,a,b]=[a,b,a,b,a,b].
\end{equation}
Moreover, by $(\ref{1-2})$ and $(\ref{1-3})$,
\begin{equation}\label{1-5}
 [a,b,b,b,a,a]=[a,b,a,a,b,b]=[a,b,a,b,a,b]^{-2}.
\end{equation}
Thus $K_{6}(G)\leq\langle[a,b,a,b,a,b]\rangle$.
\end{proof}

\begin{lem}\label{t2} Assume $G$ is a finite $p$-group and $G=\langle a,b\rangle$, where $p\geqslant 3$. If $G\in \mathcal{P}_3$, then
$$K_{5}(G)=\langle[a,b,a,b,a],~[a,b,b,a,b],~K_6(G)\rangle\lesssim {\rm C}_p^3.$$
\end{lem}

\begin{proof} Notice that $[K_5(G),K_5(G)]\leq K_{10}(G)=1$. We have $K_5(G)$ is abelian.  It suffices to show that  $K_{5}(G)\leq\langle[a,b,a,b,a],~[a,b,b,a,b],~K_6(G)\rangle$ and $[a,b,a,b,a]^p=[a,b,b,a,b]^p=1$.

Firstly, it follows from $G\in \mathcal{P}_3$ that $c(\langle [a,b],a,b^p\rangle)\leqslant 3$.
By Lemma \ref{l6}, Lemma \ref{b1}$(2)$ and Lemma \ref{t3},
\begin{equation*}
  \begin{aligned}
    1 =& [[a,b],a,b^p,a]=[[a,b,a],b^p,a]=[[a,b,a,b]^p[a,b,a,b,b]^{(^p_2)},a] \\
      =& [[a,b,a,b]^p,a] [[a,b,a,b,b]^{(^p_2)},a]=[a,b,a,b,a]^p[a,b,a,b,b,a]^{(^p_2)}\\
      =& [a,b,a,b,a]^p.
  \end{aligned}
\end{equation*}
 Similarly, $[a,b,b,a,b]^p=1$.

\medskip
Next \f by Lemma \ref{p3}, $$K_5(G)=\langle[a,b,x_1,x_2,x_3],K_6(G)\mid x_i\in\{a,b\}~\mbox{for}~1\leqslant i\leqslant3\rangle.$$
By Lemma~\ref{c1}, $$[a,b,a,a,a]=[a,b,b,b,b]=1.$$
By Lemma~\ref{t1},
\begin{equation}\label{2-1}
\begin{split}
     [a,b,b,a,a]&\equiv[a,b,a,b,a]~(\mod K_6(G)), \\
     [a,b,a,b,b]&\equiv[a,b,b,a,b]~(\mod K_6(G)).
\end{split}
\end{equation}
Since $G\in \mathcal{P}_3$, $c(\langle [a,b],ab^i\rangle)\leqslant 3$~for any integer $i$. Thus,
$[[a,b],ab^i,ab^i,ab^i]=1.$
By Lemma \ref{p2}, we have
$$([a,b,a,a,b][a,b,a,b,a][a,b,b,a,a])^i([a,b,b,b,a][a,b,b,a,b][a,b,a,b,b])^{i^2}\equiv 1~(\mod K_6(G)).$$
 Since $p\geqslant 3$, by the arbitrariness of $i$, we have
$$[a,b,a,a,b][a,b,a,b,a][a,b,b,a,a]\equiv 1~(\mod K_6(G)),$$
$$[a,b,b,b,a][a,b,b,a,b][a,b,a,b,b]\equiv 1~(\mod K_6(G)).$$
It follows from $(\ref{2-1})$ that
$$[a,b,a,a,b]\equiv [a,b,a,b,a]^{-2}~(\mod K_6(G)),$$
$$[a,b,b,b,a]\equiv [a,b,b,a,b]^{-2}~(\mod K_6(G)).$$
Thus $K_{5}(G)\leq\langle[a,b,a,b,a],~[a,b,b,a,b],~K_6(G)\rangle$.
\end{proof}

\begin{lem}\label{l31}  Assume $G$ is a finite $p$-group with $d(G)=2$ and $p\geqslant 3$. If $G\in \mathcal{P}_3$, then for any$~x,y,z,t\in G$,
$$[x^p,y,z,t]=[x,y,z,t]^p[x,y,x,x,z,t]^{(^p_3)}~{\rm and}~[y,x^p,z,t]=[y,x,z,t]^p[y,x,x,x,z,t]^{(^p_3)}.$$
In particular, if $z=x$ or $t=x$, then
$$[x^p,y,z,t]=[x,y,z,t]^p~{\rm and}~[y,x^p,z,t]=[y,x,z,t]^p.$$
\end{lem}
\begin{proof}

By Lemma~\ref{l3}, $G/K_5(G)$ is metabelian.
Then, by Lemma \ref{p1}$(2)$,
$$[x^p,y]\equiv[x,y]^p[x,y,x]^{(^p_2)}[x,y,x,x]^{(^p_3)}~(\mod K_5(G)).$$
Let $u=[x,y]^p,~v=[x,y,x]^{(^p_2)},~w=[x,y,x,x]^{(^p_3)}$. Then
$$[x^p,y,z] \equiv [uvw,z]~(\mod K_6(G)).$$
By Lemma \ref{b1}$(2)$,
$$[x^p,y,z] \equiv [u,z][u,z,vw][v,z][v,z,w][w,z]~(\mod K_6(G)).$$
Notice that $[u,z,vw],[v,z,w]\in K_6(G)$. We have

$$[x^p,y,z] \equiv [u,z][v,z][w,z]~(\mod K_6(G)).$$
That is
$$[x^p,y,z] \equiv [[x,y]^p,z][[x,y,x]^{(^p_2)},z][[x,y,x,x]^{(^p_3)},z]~(\mod K_6(G)).$$
By Lemma~\ref{l6},
$$[x^p,y,z] \equiv [x,y,z]^p[x,y,z,[x,y]]^{(^p_2)}[x,y,x,z]^{(^p_2)}[x,y,x,x,z]^{(^p_3)}~(\mod K_6(G)).$$
It follows that
$$[x^p,y,z,t]=[[x,y,z]^p[x,y,z,[x,y]]^{(^p_2)}[x,y,x,z]^{(^p_2)}[x,y,x,x,z]^{(^p_3)},t].$$
Similar to the above calculation, we have
$$[x^p,y,z,t]=[x,y,z,t]^p[x,y,z,[x,y],t]^{(^p_2)}[x,y,x,z,t]^{(^p_2)}[x,y,x,x,z,t]^{(^p_3)}.$$
Since $p\geqslant 3$, by Lemma \ref{t3} and Lemma \ref{t2},
$$[x^p,y,z,t]=[x,y,z,t]^p[x,y,x,x,z,t]^{(^p_3)}.$$
Similarly,
$$[y,x^p,z,t]=[y,x,z,t]^p[y,x,x,x,z,t]^{(^p_3)}.$$
If $z=x$ or $t=x$, then, by Lemma \ref{c1},
$$[x,y,x,x,z,t]=[y,x,x,x,z,t]=1.$$
Thus,
$$[x^p,y,z,t]=[x,y,z,t]^p~{\rm and}~[y,x^p,z,t]=[y,x,z,t]^p.$$
\end{proof}

\noindent {\bf Proof of Theorem.}
Let $G=\langle a,b\rangle$.
It suffices to show that $K_6(G)=1$. 

\medskip
Assume $p=2$.
Let $H=\langle a^2,b,[a,b]\rangle$. Then $H<G$. Since $G\in \mathcal{P}_3$, $c(H)\leqslant 3$. Then by Lemma \ref{b1}$(2)$,
$$1=[a^2,b,[a,b],b]=[[a,b][a,b,a][a,b],[a,b],b]=[a,b,a,[a,b],b].$$
Let $x=[a,b,a]$. Then $[x,[a,b],b]=1$. From Witt Formula and Lemma \ref{b1}$(1)$,
\begin{equation*}
  \begin{aligned}
1=&[x,a,b^x][b,x,a^b][a,b,x^a]\\
=&[x,a,b[b,x]][b,x,a[a,b]][a,b,x[x,a]]\\
=&[x,a,[b,x]][x,a,b][x,a,b,[b,x]][b,x,[a,b]][b,x,a]\\
&[b,x,a,[a,b]][a,b,[x,a]][a,b,x][a,b,x,[x,a]].
  \end{aligned}
\end{equation*}
Notice that $[x,a,[b,x]]=[x,a,b,[b,x]]=[b,x,a,[a,b]]=[a,b,x,[x,a]]=1$. We have
\begin{equation*}
  \begin{aligned}
1=&[x,a,b][b,x,[a,b]][b,x,a][a,b,[x,a]][a,b,x]\\
=&[x,a,b][b,x,[a,b]][b,x,a][[x,a],[a,b]]^{-1}[a,b,x].
  \end{aligned}
\end{equation*}
Since $G\in \mathcal{P}_3$, $c(\langle [a,b],a\rangle)\leqslant 3$. Thus
$$[[a,b],a,a,[a,b]]=1,~ {\rm that~ is}~ [[x,a],[a,b]]=1.$$
Moreover,
$$1=[x,a,b][b,x,[a,b]][b,x,a][a,b,x].$$
And then, by Lemma \ref{b1}$(2)$,
\begin{equation*}
  \begin{aligned}
  1 =&[[x,a,b][b,x,[a,b]][b,x,a][a,b,x],b]\\
    =&[x,a,b,b][b,x,[a,b],b][b,x,a,b][a,b,x,b].
  \end{aligned}
\end{equation*}
By Lemma \ref{b1}$(3)$, $[a,b,x,b]=[x,[a,b],b]^{-1}=1$.
Notice that $[b,x,[a,b],b]=1$.  We have
$$1=[x,a,b,b][b,x,a,b]=[x,a,b,b][x,b,a,b]^{-1}.$$
Thus $[x,a,b,b]=[x,b,a,b]$. That is,
$$[a,b,a,a,b,b]=[a,b,a,b,a,b].$$
It follows from formula $(\ref{1-5})$ that $[a,b,a,b,a,b]=1$.
Then $K_6(G)=1$ by Lemma~\ref{t3}.

\medskip
Assume $p=3$.
Since $G\in \mathcal{P}_3$, by formulae $(\ref{1-4})$, $(\ref{1-5})$ and Lemma \ref{t3}, we have
\begin{equation}\label{7}
\begin{split}
     [a,b,a,a,b,b]&=[a,b,b,b,a,a]=[a,b,a,b,a,b] \\
     =[a,b,a,b,b,a]&=[a,b,b,a,a,b]=[a,b,b,a,b,a].
\end{split}
\end{equation}
Let $H=\langle a^3,ab\rangle$. Then $H<G$. Since $G\in \mathcal{P}_3$, $c(H)\leqslant 3$. Thus $[a^3,ab,ab,ab]=1$. By Lemma \ref{b1}$(1)(2)$,
\begin{equation}\label{4-1}
  \begin{aligned}
    1 =& [a^3,ab,ab,ab]=[a^3,b,ab,ab]=[[a^3,b,b][a^3,b,a][a^3,b,a,b],ab] \\
      =& [a^3,b,b,b][a^3,b,a,b][a^3,b,a,b,b][a^3,b,b,a][a^3,b,a,a][a^3,b,a,b,a]\\
      &[a^3,b,b,a,b][a^3,b,a,a,b][a^3,b,a,b,a,b].
  \end{aligned}
\end{equation}
Since $G\in \mathcal{P}_3$, $c(\langle a^3,b\rangle)\leqslant 3$. Thus
\begin{equation*}\label{4-2}
[a^3,b,b,b]=1.
\end{equation*}
By Lemma \ref{l31},
$$[a^3,b,a,b]=[a,b,a,b]^3,~[a^3,b,b,a]=[a,b,b,a]^3~{\rm and}~ [a^3,b,a,a]=[a,b,a,a]^3.$$
Thus, by Lemma \ref{l6} and Lemma \ref{t2},
$$[a^3,b,a,b,b]=[[a,b,a,b]^3,b]=[a,b,a,b,b]^3=1.$$
Similarly,
$$[a^3,b,a,b,a]=[a^3,b,b,a,b]=[a^3,b,a,a,b]=1.$$
Moreover,
$$[a^3,b,a,b,a,b]=1.$$
Hence by $(\ref{4-1})$, we have
$$[a,b,a,b]^3[a,b,b,a]^3[a,b,a,a]^3=1.$$
By Lemma \ref{t1} and Lemma \ref{t2}, $[a,b,a,b]^3=[a,b,b,a]^3$. Thus
$$[a,b,a,b]^6[a,b,a,a]^3=1.$$
Similarly, let $K=\langle a^3,ab^{-1}\rangle$. We have
$$[a,b,a,b]^6[a,b,a,a]^{-3}=1.$$
Thus, $[a,b,a,b]^3=1$ and $[a,b,a,a]^3=1$.
Let $T=\langle a,b^3\rangle$. Then $T<G$. Since $G\in \mathcal{P}_3$, $c(T)\leqslant3$. It follows from Lemma \ref{l31} that
$$1=[a,b^3,a,a]=[a,b,a,a]^3[a,b,b,b,a,a].$$
Thus $[a,b,b,b,a,a]=1$. Then we have $K_6(G)=1$ by (\ref{7}) and Lemma \ref{t3}.\qed

\medskip
Finally, we give examples to show that the bound for the class of the $p$-groups in Theorem is best.

\begin{exam}\label{e3}  {\rm Let $G=\lg a,b\mid a^{2^n}=b^{2^3}=c^{2^2}=d_1^{2^2}=d_2^{2^2}=e_1^{2^2}=e_2^{2}=1, ~[a,b]=c, ~[c,a]=d_1, ~[c,b]=d_2,~[d_1,a]=e_1,~[d_1,b]=e_2,~[d_2,a]=e_1^2e_2,~[d_2,b]=1, ~[e_2,a]=e_1^2, ~[e_1,a]=[e_1,b]=[e_2,b]=[e_3,a]=[e_3,b]=1\rg$, where $n\geqslant3$. It is straightforward to get that $c(G)=5$ and $G\in\mathcal{P}_3$.}
\end{exam}

\begin{exam}\label{e4}  {\rm Let $G=\lg a,b\mid a^{3^n}=b^3=c^3=d_1^{3}=d_2^{3}=e_1^{3}=e_2^{3}=f^3=1, ~[a,b]=c, ~[c,a]=d_1, ~[c,b]=d_2,
~[d_1,a]=e_1, ~[d_1,b]=[d_2,a]=e_2, ~[d_2,b]=1, ~[e_1,b]=[e_2,a]=f, ~[e_1,a]=[e_2,b]=[f,a]=[f,b]=1\rg$, where $n\geqslant2$. It is straightforward to get that $c(G)=5$ and $G\in\mathcal{P}_3$.}
\end{exam}


\begin{thebibliography}{999}

\bibitem{ES}
M.J. Evans and B.G. Sandor, {Groups of class $2n$ in which all proper subgroups
have class at most $n$}, {\it J. Algebra}, {\bf 498}(2018), 165-177.

\bibitem{CK}
C.K. Gupta, {A bound for the class of certain nilpotent groups}, {\it J. Austral. Math. Soc.,} {\bf 5}(1965), 506-511.

\bibitem{Hu}
B. Huppert,  {\em Endliche Gruppen I},  { New York-Berlin, Springer}, 1967.


\bibitem{B2}
P.J. Li, H.P. Qu and J.W. Zeng, {Finite $p$-groups whose proper subgroups are of class $\leqslant n$}, {\it J. Algebra Appl.,} {\bf16:1}(2017), 1750014, 8pp.

\bibitem{p>5}
I.D. Macdonald, Generalizations of a classical theorem about nilpotent groups, {\it Illinois J. Math.,} {\bf 8}(1964), 556-570.

\bibitem{NW}
M.F. Newman and James Wiegold, Groups with many nilpotent subgroups, {\it Arch. Math.,} {\bf 15}(1964), 241-250.

\bibitem{minimal nonabelian}
L. R\'{e}dei, Das schiefe Produkt in der Gruppentheorie, {\it Comment. Math. Helv.}, {\bf 20}(1947), 225-264.

\bibitem{Xu}
M.Y. Xu, {A theorem on metabelian $p$-groups and some consequences}, {\it Chinese Ann. Math. Ser. B,} {\bf5:1}(1984), 1-6. 

\bibitem{Qu}
M.Y. Xu and H.P. Qu, {\em Finite $p$-groups}, { Peking University Press}, 2010.
(in Chinese)

\bibitem{Za}
D.I. Zaitsev, Stably nilpotent groups, {\it Mat. Zametki}, {\bf 2}(1967), 337-346. English translation in: {\it Math. Notes}, {\bf 2:4}(1967), 690-694.

\end{thebibliography}
\end{document}